\documentclass[a4paper,9pt,twoside]{article}
\usepackage{latexsym}
\usepackage{a4wide}
\usepackage{amscd}
\usepackage{graphics}
\usepackage{graphicx}
\usepackage[english]{varioref}
\usepackage{amsmath}
\usepackage{amssymb}
\usepackage{mathrsfs}
\usepackage{amsthm}
\usepackage{amssymb,tikz}
\usepackage{verbatim}
\usepackage{soul}
\usepackage{lipsum}

\usepackage{bbm}
\usepackage{color}
\usepackage{accents}
\usepackage{enumerate}
\usepackage[utf8x]{inputenc}
\usepackage[T1]{fontenc}
\usepackage[english]{babel}
\usepackage{amsfonts, amscd, amsmath, amssymb, amsthm, mathrsfs, mathtools}
\usepackage{braket}
\input xy
\xyoption{all}
\usepackage[utf8x]{inputenc}
\usepackage[T1]{fontenc}
\usepackage[english]{babel}
\usepackage{amsfonts, amscd, amsmath, amssymb, amsthm, mathrsfs, mathtools}
\usepackage{graphicx}
\usepackage{subfig}
\usepackage{braket}
\usepackage[left=2.5cm,right=2.5cm,top=2.5cm,bottom=2.5cm]{geometry}
\usepackage{fancyhdr}
\theoremstyle{definition}

\newtheorem{remark}{Remark}[section]

\newtheorem{esempio}{Example}[section]

\theoremstyle{plain}
\newtheorem{definizione}{Definition}[section]
\newtheorem{teorema}{Theorem}[section]
\newtheorem{proposizione}{Proposition}[section]
\newtheorem{lemma}{Lemma}[section]
\newtheorem{corollario}{Corollary}[section]

\newcommand{\numberset}{\mathbb}

\newcommand{\R}{\numberset{R}}
\newcommand{\N}{\numberset{N}}
\newcommand{\Z}{\numberset{Z}}

\usepackage{bookmark,hyperref}

\DeclarePairedDelimiter{\abs}{\lvert}{\rvert}
\DeclarePairedDelimiter{\norma}{\lVert}{\rVert}

\makeatletter
\let\oldabs\abs
\def\abs{\@ifstar{\oldabs}{\oldabs*}}
\let\oldnorma\norma
\def\norma{\@ifstar{\oldnorma}{\oldnorma*}}

\title{\bf{The generalized recurrent set, \\ explosions and Lyapunov functions}}

\author{\hspace{-1cm}{OLGA BERNARDI${\,
}^{1}$ \quad ANNA FLORIO${\,
}^{2}$  \quad JIM WISEMAN${\,
}^{3}$                }
\vspace{0.3cm}
\\
\hspace{-1cm}${\ }^{1}$ Dipartimento di Matematica ``Tullio Levi-Civita'',
Universit\`a di Padova,  Italy
\\
\\
\hspace{-1cm}${\ }^{2}$ Laboratoire de Mathématiques d'Avignon, Avignon Université, France \\
\\
\hspace{-1cm}${\ }^{3}$ Department of Mathematics, Agnes Scott College, Decatur, Georgia, USA \\}

\date{ }

\begin{document}
\maketitle


\begin{abstract} 
\noindent We consider explosions in the generalized recurrent set for homeomorphisms on a compact metric space. We provide multiple examples to show that such explosions can occur, in contrast to the case for the chain recurrent set. We give  sufficient conditions  to avoid explosions and  discuss their necessity. Moreover, we explain the relations between explosions and  cycles for the generalized recurrent set. In particular, for a compact topological manifold with dimension greater or equal $2$, we characterize explosion phenomena in terms of existence of cycles. We apply our results to give sufficient conditions for stability, under $\mathscr{C}^0$ perturbations, of the property of admitting a continuous Lyapunov function which is not a first integral. \\ 
\end{abstract}

\section{Introduction}
Generalized recurrence was originally introduced for flows by Auslander in the Sixties \cite{AUS} by using continuous Lyapunov functions. Auslander defined the generalized recurrent set to be the union of those orbits along which all continuous Lyapunov functions are constant. In the same paper, he gave a characterization of this set in terms of the theory of prolongations. The generalized recurrent set was later extended to maps by Akin and Auslander (see \cite{A} and \cite{AA}). More recently Fathi and Pageault \cite{FP15} proved that, for a homeomorphism, the generalized recurrent set can be equivalently defined by using Easton's strong chain recurrence \cite{EA78}. \\
~\newline
The present paper is concerned with the behaviour, under continuous perturbations of the map, of the generalized recurrent set for homeomorphisms. In particular, we analyze the phenomenon of explosions, which are  discontinuous jumps in the size of the generalized recurrent set.  Moreover, we apply our results to give  sufficient conditions  to assure the persistence under continuous perturbations of a continuous Lyapunov function which is not a first integral (that is a continuous strict Lyapunov function). \\
~\newline
Throughout the paper, $(X,d)$ is a compact metric space. We denote by $Hom(X)$ the space of homeomorphisms of $X$ equipped with the uniform topology induced by the metric: 
$$d_{\mathscr{C}^0}(f,g) = \max_{x \in X} d(f(x),g(x)).$$
Let $\mathcal{GR}(f)$ be the generalized recurrent set of $f \in Hom(X)$ (the rigorous definition will be given in Section \ref{pre}). \\
~\newline
{\textit{Assume that $\mathcal{GR}(f) \ne X$. We say $f$ does not permit $\mathcal{GR}$-explosions if for any open neighborhood $U$ of $\mathcal{GR}(f)$ in $X$ there exists a neighborhood $V$ of $f$ in $Hom(X)$ such that if $g \in V$ then $\mathcal{GR}(g) \subset U$. \\
We say $f$ does not permit $\mathcal{GR}$-full explosions if there exists a neighborhood $V$ of $f$ in $Hom(X)$ such that if $g \in V$ then $\mathcal{GR}(g) \ne X$. }}\\
~\newline
Explosions have been studied  for both the non-wandering set $\mathcal{NW}(f)$ 
and the chain recurrent set $\mathcal{CR}(f)$
(see e.g.\ \cite{SS} and \cite{BF85} respectively).  We introduce the notion of $\mathcal{GR}$-full explosion for the application to continuous strict Lyapunov functions. After some preliminary results and definitions, in Section \ref{pre} we observe that $\mathcal{GR}$-(full) explosions in general can occur (Example \ref{cantor-fissi}). \\
~\newline
Section \ref{tre} gives sufficient conditions to avoid $\mathcal{GR}$-explosions (Corollary \ref{COR UNO} and Proposition \ref{top stab imply no GR exp}) and $\mathcal{GR}$-full explosions (Proposition \ref{anna}). These results can be summarized as follows. \\
~\newline
{\textit{Assume that $\mathcal{GR}(f) \ne X$. If $(i)$ $\mathcal{GR}(f) = \mathcal{CR}(f)$ or $(ii)$ $f$ is topologically stable then $f$ does not permit $\mathcal{GR}$-explosions. If $\mathcal{CR}(f) \ne X$ then $f$ does not permit $\mathcal{GR}$-full explosions.}} \\
~\newline
The core of Section \ref{tre} concerns the necessity of these conditions. In particular, in Examples \ref{wiseman}, \ref{attrattore-repulsore}  and \ref{example JW 2 revised} respectively, we show that the converses of the previous results are in general false on compact metric spaces. \\
~\newline
The goal of Section \ref{quattro} is to explain the relations between explosions and  cycles for the generalized recurrent set. The notion of cycle will be rigorously recalled in Definition \ref{cicli}. Since we need to apply the $\mathscr{C}^0$ closing lemma, the ambient space is a compact, topological manifold $M$ with $dim(M) \ge 2$. \\
~\newline
\textit{Assume that $\mathcal{GR}(f) \ne M$. $f$ does not permit ${\mathcal{GR}}$-explosions if and only if there exists a decomposition of ${\mathcal{GR}}(f)$ with no cycles.} \\
~\newline
The above theorem generalizes the corresponding result for the non-wandering set. More precisely, the fact that the existence of a decomposition of $\mathcal{NW}(f)$ without cycles prohibits $\mathcal{NW}$-explosions is due to Pugh and Shub (see Theorem 6.1 in \cite{PS} for flows and Theorem 5.6 in \cite{WEN} for homeomorphisms); the converse has been proved by Shub and Smale (see Lemma 2 in \cite{SS}). We postpone the proofs of both implications to Appendixes \ref{APP1} and \ref{APP}. \\
We remark that, in the proof of this theorem, we do not use the full definition of the generalized recurrent set, but only the fact that $\mathcal{GR}(f)$ is a compact, invariant set, which contains $\mathcal{NW}(f)$ and is contained in $\mathcal{CR}(f).$ Consequently, with the same proof, we obtain the corresponding results both for the strong chain recurrent set $\mathcal{SCR}_d(f)$ and $\mathcal{CR}(f)$, see Theorem \ref{teo 4} and Remark \ref{rem CR}. \\
~\newline
In Section \ref{cinque} we apply the results of Sections \ref{tre} and \ref{quattro} to obtain sufficient conditions for an affirmative answer to the question: \\
~\newline
\textit{For a given homeomorphism, is the property of admitting a continuous strict Lyapunov function stable under $\mathscr{C}^0$ perturbations?} \\
~\newline
We refer to Propositions \ref{Lyapunov stability} and \ref{per varieta}. Finally, in Proposition \ref{Lyapunov genericity}, we remark that on a smooth, compact manifold the property of admitting such a function is generic.

\section{Preliminaries} \label{pre}
\noindent In this section we recall the notions of chain recurrent, strong chain recurrent and generalized recurrent point for a fixed $f \in Hom(X)$. \\
\indent Given $x,y\in X$ and $\varepsilon>0$, an \emph{$\varepsilon$-chain} from $x$ to $y$ is a finite sequence $(x_i)_{i=1}^{n}\subset X$ such that $x_1=x$ and, setting, $x_{n+1}=y$, we have
\begin{equation} \label{prima}
d(f(x_i),x_{i+1}) < \varepsilon \qquad \forall i=1,\ldots,n.
\end{equation}
A point $x\in X$ is said to be \emph{chain recurrent} if for all $\varepsilon>0$ there exists an $\varepsilon$-chain from $x$ to $x$. The set of chain recurrent points is denoted by $\mathcal{CR}(f)$.
\noindent Since we assumed $X$ to be compact, chain recurrence depends only on the topology, not on the choice of the metric (see for example  \cite{DeV}[Theorem 4.4.5] and \cite{FR}[Section 1]). \\
 \indent Given $x,y\in X$ and $\varepsilon>0$, a \emph{strong $\varepsilon$-chain} from $x$ to $y$ is a finite sequence $(x_i)_{i=1}^{n}\subset X$ such that $x_1=x$ and, setting $x_{n+1}=y$, we have
\begin{equation} \label{seconda}
\sum_{i=1}^n d(f(x_i),x_{i+1}) < \varepsilon.
\end{equation}
A point $x\in X$ is said to be \emph{strong chain recurrent} if for all $\varepsilon>0$ there exists a strong $\varepsilon$-chain from $x$ to $x$. The set of strong chain recurrent points is denoted by $\mathcal{SCR}_d(f)$.
\noindent In general, strong chain recurrence depends on the choice of the metric; see for example \cite{YO}[Example 3.1] and \cite{WI}[Example 2.6]. \\
\indent A way to eliminate the dependence on the metric in $\mathcal{SCR}_d(f)$ is taking the intersection over all metrics. We then obtain the generalized recurrent set
\begin{equation} \label{terza}
\mathcal{GR}(f) := \bigcap_{d} \mathcal{SCR}_d(f)
\end{equation}
where the intersection is over all metrics compatible with the topology of $X$. 
\\
\noindent The sets $\mathcal{GR}(f)$, $\mathcal{SCR}_d(f)$ and $\mathcal{CR}(f)$ are all closed and invariant (see respectively \cite{AA}[Page 52], \cite{FP15}[Page 1193] and \cite{FR}[Section 1]). Moreover, in general $\mathcal{NW}(f) \subset \mathcal{GR}(f) \subset \mathcal{SCR}_d(f) \subset \mathcal{CR}(f)$, where $\mathcal{NW}(f)$ denotes the non-wandering set of $f$, and all inclusions can be strict. We refer to \cite{WI}[Example 2.9] for an exhaustive treatment of these inclusions. \\
\indent The dynamical relevance of the generalized recurrent set relies on its relations with continuous Lyapunov functions. A continuous function $u:X\rightarrow\R$ is a Lyapunov function for $f$ if $u(f(x))\leq u(x)$ for every $x\in X$. Given a Lyapunov function $u:X \to \R$ for $f$, the corresponding neutral set is given by
$$\mathcal{N}(u) = \{x \in X: \ u(f(x)) = u(x)\}.$$
We refer to \cite{FP15}[Theorem 3.1] for the proof of the next result.
\begin{teorema} \label{generalized}
$$\mathcal{GR}(f) = \bigcap_{u \in \mathscr{L}(f)} \mathcal{N}(u),$$
where $\mathscr{L}(f)$ is the set of continuous Lyapunov functions for $f$. Moreover, there exists a continuous Lyapunov function for $f$ such that $\mathcal{N}(u) = \mathcal{GR}(f)$.   
\end{teorema}
\noindent In order to describe the behavior of $\mathcal{GR}(f)$ under continuous perturbations of $f$, we introduce and discuss the notions of $\mathcal{GR}$-explosion and $\mathcal{GR}$-full explosion. \\
\indent We start by recalling the phenomenon of explosions for the generalized recurrent set, which are particular discontinuities of the function 
$$Hom(X) \ni f \mapsto \mathcal{GR}(f) \in \mathcal{P}(X),$$
where $\mathcal{P}(X)$ denotes the power set of $X$. 
\begin{definizione} \label{la prima} (No $\mathcal{GR}$-explosions) Let $f \in Hom(X)$ be such that $\mathcal{GR}(f) \ne X$. We say $f$ does not permit $\mathcal{GR}$-explosions if for any open neighborhood $U$ of $\mathcal{GR}(f)$ in $X$ there exists a neighborhood $V$ of $f$ in $Hom(X)$ such that if $g \in V$ then $\mathcal{GR}(g) \subset U$.
\end{definizione}
\noindent No $\mathcal{NW}$-explosions and no $\mathcal{CR}$-explosions are defined analogously (see for example \cite{SS}[Page 588] for $\mathcal{NW}(f)$ and \cite{BF85}[Page 323] for $\mathcal{CR}(f)$). 
We recall that $\mathcal{NW}$-explosions in general can occur; see \cite{palis}, \cite{NI71}[Section 6.3] and \cite{WEN}[Section 5.2]. This is not the case for $\mathcal{CR}$-explosions; see the following, which is Theorem F in \cite{BF85}.
\begin{teorema} \label{teo 1} Let $f \in Hom(X)$ be such that $\mathcal{CR}(f) \ne X$. Then $f$ does not permit $\mathcal{CR}$-explosions. 
\end{teorema} 
\noindent In \cite{BF85}, the proof of the previous result essentially uses a dynamical characterization of the points outside $\mathcal{CR}(f)$. We propose an alternative, direct  proof of this fact. \\
~\newline
\textit{Proof of Theorem \ref{teo 1}.} Argue by contradiction and suppose there are an open neighborhood $U$ of $\mathcal{CR}(f)$ in $X$, a sequence $(g_n)_{n \in \mathbb{N}} \in Hom(X)$ converging to $f$ in the uniform topology and a sequence of points $(y_n)_{n \in \mathbb{N}}$ such that $y_n \in \mathcal{CR}(g_n) \setminus U$. Since $X$ is compact and $U$ is open, we can assume that the sequence $(y_n)_{n\in\N}$ converges to $y\notin U$. In particular, $y$ does not belong to $\mathcal{CR}(f)$. \\
\noindent Let $\varepsilon > 0$ be fixed. By hypothesis, for any $n \in \mathbb{N}$ there exists an $\frac{\varepsilon}{3}$-chain $(x_1=y_n,x_2,\dots,x_m,x_{m+1}=y_n)$ for $g_n$ from $y_n$ to $y_n$. 
Corresponding to $\varepsilon > 0$, there exists $\bar{n} \in \N$ such that 
\begin{equation}\label{thmCRnoEX 1}
d_{\mathscr{C}^0}(f,g_n) < \dfrac{\varepsilon}{3} \qquad \forall n\geq\bar{n}.
\end{equation}
Moreover, by the uniform continuity of $f$, there exists $\delta\in(0,\frac{\varepsilon}{3})$ such that
\begin{equation}\label{thmCRnoEX 2}
d(z,w) < \delta\qquad\Rightarrow\qquad d(f(z),f(w)) < \dfrac{\varepsilon}{3}.
\end{equation}
Finally, let $\tilde{n}\in\N$ such that 
\begin{equation}\label{thmCRnoEX 3}
d(y,y_n) < \delta \qquad \forall n\geq\tilde{n}.
\end{equation}
We will show that, if $n \ge \max(\bar{n},\tilde{n})$ then the chain 
$$(y,x_2,\dots,x_m,y)$$ 
is an $\varepsilon$-chain for $f$ from $y$ to $y$. Indeed, thanks to inequalities \eqref{thmCRnoEX 1},\eqref{thmCRnoEX 2} and \eqref{thmCRnoEX 3} and the fact that $(x_1=y_n,x_2,\dots,x_m,x_{m+1}=y_n)$ is an $\frac{\varepsilon}{3}$-chain for $g_n$, we have 
$$
d(f(y),x_2)\leq d(f(y),f(y_n))+d(f(y_n),g_n(y_n))+d(g_n(y_n),x_2) < \varepsilon.
$$
\\
\noindent Moreover
$$
d(f(x_i),x_{i-1})\leq d(f(x_i),g_n(x_i))+d(g_n(x_i),x_{i+1}) < \frac{2 \varepsilon}{3} < \varepsilon$$
for all $i=2, \ldots,m-1$. \noindent Finally
$$
d(f(x_m),y)\leq d(f(x_m),g_n(x_m))+d(g_n(x_m),y_n)+d(y_n,y) < \varepsilon.
$$
By the arbitrariness of $\varepsilon > 0$, we conclude that $y \in \mathcal{CR}(f)$, obtaining the desired contradiction. \hfill $\Box$ \\
~\newline
\indent We now introduce full explosions for the generalized recurrent set. 
\begin{definizione} \label{la seconda} (No $\mathcal{GR}$-full explosions) Let $f \in Hom(X)$ be such that $\mathcal{GR}(f) \ne X$. We say $f$ does not permit $\mathcal{GR}$-full explosions if there exists a neighborhood $V$ of $f$ in $Hom(X)$ such that if $g \in V$ then $\mathcal{GR}(g) \ne X$.
\end{definizione}
\noindent Clearly, if $f$ does not permit $\mathcal{GR}$-explosions then $f$ does not permit $\mathcal{GR}$-full explosions; if $f$ permits $\mathcal{GR}$-full explosions then $f$ permits $\mathcal{GR}$-explosions. \\
~\newline
\indent {\color{black}{We  observe that}}, unlike the chain recurrent case, $\mathcal{GR}$-(full) explosions can in general occur.

\begin{esempio} \label{cantor-fissi} 

On the circle $\mathbb{S}^1$ with the usual topology, consider an interval $I\subsetneq \mathbb{S}^1$ of positive Lebesgue measure. Let $\phi:\mathbb{S}^1\rightarrow[0,+\infty)$ be a non-negative smooth function whose set of zeroes is $I$. Let $f:\mathbb{S}^1\rightarrow\mathbb{S}^1$ be the time-one map of the flow of the vector field $$ V(x)=\phi(x)\dfrac{\partial}{\partial x}.$$
(See Figure \ref{figura3.1}).
In such a case $$ I=\mathcal{GR}(f).$$ 
We observe that, for an arbitrarily small $\varepsilon>0$, there always exists $g_{\varepsilon}\in Hom(\mathbb{S}^1)$ such that $$d_{\mathscr{C}^0}(f,g_{\varepsilon})<\varepsilon$$ and $$\mathbb{S}^1=\mathcal{GR}(g_{\varepsilon}).$$ 
(Simply perturb $\phi$ so that it is positive on $I$). Consequently, $f\in Hom(\mathbb{S}^1)$ is an example of a homeomorphism which permits $\mathcal{GR}$-full explosions. In particular, $f$ admits $\mathcal{GR}$-explosions. 
\begin{figure}[h]
\centering
\includegraphics[scale=.2]{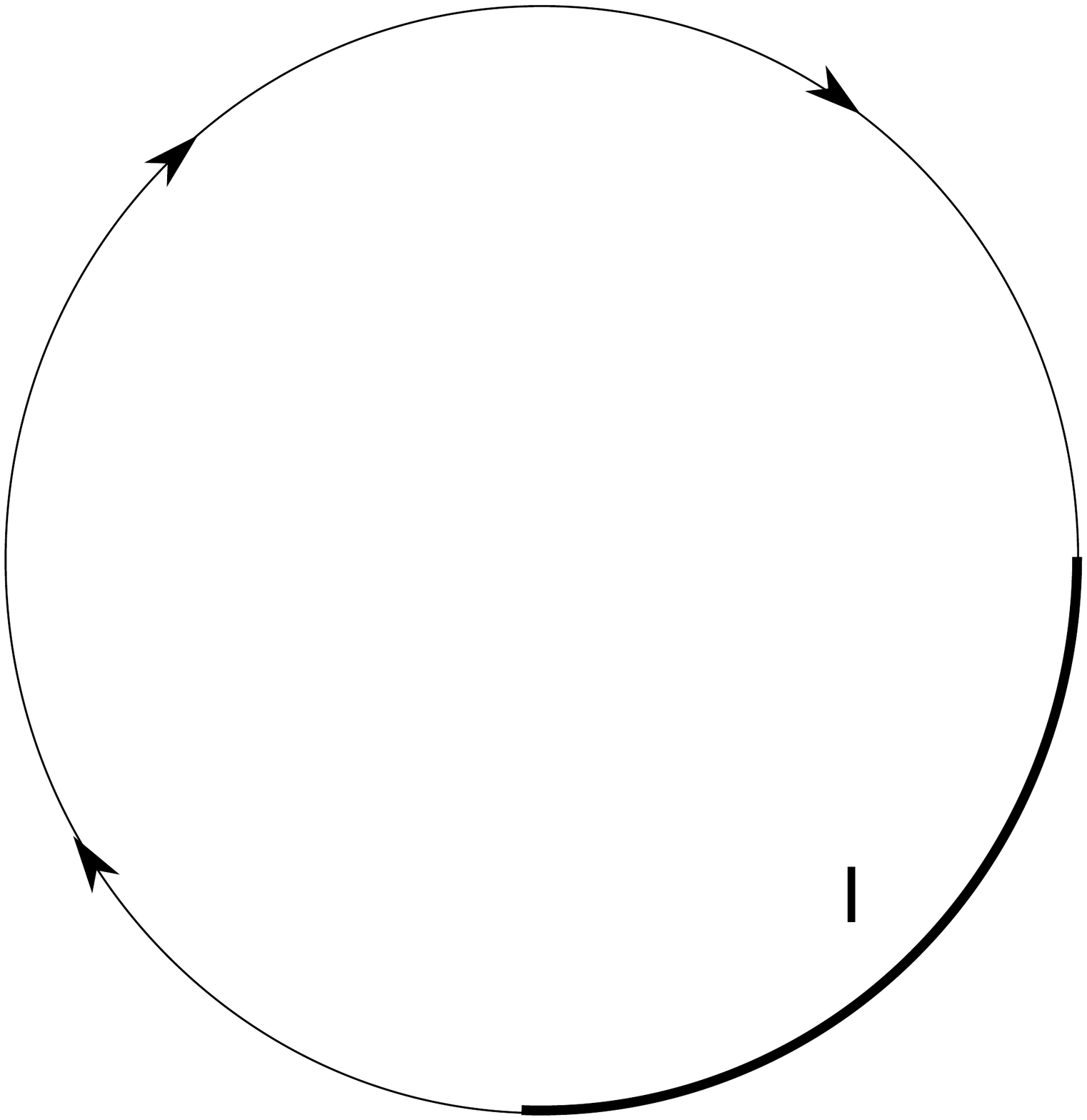}
\caption{The dynamical system of Example \ref{cantor-fissi}.}
\label{figura3.1}
\end{figure}
\end{esempio}

\section{$\mathcal{GR}$-(full) explosions} \label{tre}
The aim of this section is to discuss some sufficient conditions to avoid $\mathcal{GR}$-(full) explosions. \\
~\newline
\indent The first one comes from a straightforward application of Theorem \ref{teo 1} (see Corollary G in \cite{BF85} for $\mathcal{NW}(f)$).
\begin{corollario} \label{COR UNO}
Let $f \in Hom(X)$ be such that $\mathcal{GR}(f) \ne X$. If $\mathcal{GR}(f) = \mathcal{CR}(f)$ then $f$ does not permit $\mathcal{GR}$-explosions. 
\end{corollario}
\noindent Notice that the converse of Corollary \ref{COR UNO} is  false, in general. In the next two examples, we define homeomorphisms which do not permit $\mathcal{GR}$-explosions even though $\mathcal{GR}(f) \ne \mathcal{CR}(f)$.
\begin{esempio}\label{example 2 JW}
Let $\mathbb{S}^1$ be the circle with the usual topology. We consider it as the interval $I = [0,1]$ with the endpoints identified. Let $K\subset\mathbb{S}^1$ be the middle-third Cantor set constructed on the interval $I$. Denote as $\{e_n\}_{n\in\N}\subset K$ the set of endpoints of the removed intervals. At each $e_n$, glue $n$ copies of the interval $I$. Let $X$ be the union of $\mathbb{S}^1$ with all these attached copies of $I$. Define the homeomorphism $f:X\rightarrow X$ as follows: $f$ fixes $K$ and every copy of the interval $I$, $f$ moves all the points in $\mathbb{S}^1\setminus K$ counterclockwise (see Figure \ref{figura en}). Clearly, it holds that
$$\mathcal{CR}(f)=X.$$
Moreover, since there is a Cantor set of fixed points, 
$$\mathcal{GR}(f)=Fix(f).$$
We refer to \cite{FP15}[Example 3.3] for details on this argument. The idea is that, even if $K$ has vanishing Lebesgue measure, the dynamical system is topologically conjugate to the case $\lambda_{\text{Leb}}(K) > 0$, in which case strong chains cannot cross $K$. \\ 
\noindent We observe that every $g \in Hom(X)$ must fix $K$. Indeed, since the $e_n$'s have homeomorphically distinct neighborhoods, any homeomorphism $g$ must fix each $e_n$. Moreover, since the endpoints are dense in $K$, $g$ must fix the entire Cantor set. As a consequence, the homeomorphism $f$ does not permit $\mathcal{GR}$-explosions. 
\begin{figure}[h]
\centering
\includegraphics[scale=.07]{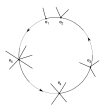}
\caption{The dynamical system of Example \ref{example 2 JW}.}
\label{figura en}
\end{figure}
\end{esempio}
\begin{esempio} \label{wiseman} 
We construct the compact metric space $X$ as follows. Let $\mathbb{S}^1$ be the circle with the usual topology. For notational convenience, we consider it as the interval $[-2,2]$ with the endpoints identified. Let $K \subset \mathbb{S}^1$ be the middle-third Cantor set constructed on the interval $[-1,1]$.  Denote the set $K \cap [0,1]$ by $K_+$ and the set $K \cap [-1,0]$ by $K_-$. Observe that the element of $K_+$ closest to $0$ is $1/3$ and the element of $K_-$ closest to $0$ is $-1/3$. For each point $k\in K_+$, let $X_{-k,k}$ be an arc (disjoint from $\mathbb{S}^1$) connecting $-k$ in $K_-$ to $k \in K_+$. Assume that these arcs are pairwise disjoint. Finally, let 
$$X = \mathbb{S}^1 \cup \left(\bigcup_{k\in K_+} X_{-k,k}\right).$$
Define the homeomorphism $f:X \to X$ as follows.  On $\mathbb{S}^1$, $f$ fixes $K$ and moves all other points counterclockwise, that is, in the direction of decreasing $x$. On every arc $X_{-k,k}$, the endpoints $-k$ and $k$ are fixed, since they are in $K$, and $f$ moves all other points from $-k \in K_-$ toward $k \in K_+$. See Figure \ref{Wiseman idea}. \\ 
In such a case,
$$\mathcal{CR}(f) = X.$$
Moreover, since there is a Cantor set of fixed points, 
$$\mathcal{GR}(f) = K \cup (-1/3,1/3) \cup X_{-1/3,1/3}.$$  
We claim that $f$ does not permit $\mathcal{GR}$-explosions. To see this, let $g$ be a homeomorphism close to $f$. Then $g(K) = K$ and each $X_{-k,k}$ maps to some $X_{-k',k'}$. Consequently, if $g$ moves $k \in K_+$ clockwise then $g$ moves $-k \in K_-$ counterclockwise; if $g$ moves $k \in K_+$ counterclockwise then $g$ moves $-k \in K_-$ clockwise. 
For any $\alpha > 0$, we can assume $g$ close enough to $f$ that 
\begin{itemize}
\item[$(a)$] $g$ moves counterclockwise any $x \in \mathbb{S}^1$ that is not within $\alpha$ of $K$;
\item[$(b)$] For any $x\in X_{-k,k}$ that is not within $\alpha$ of $K$, $g(x)$ is closer to $k \in K_+$ than is $x$.
\end{itemize}
We will show that any $x \in X$ not within $\alpha$ of $\mathcal{GR}(f)$ is not generalized recurrent. Since $\alpha$ is arbitrary, this means that $f$ does not permit $\mathcal{GR}$-explosions. \\
For arbitrary $a,b \in \mathbb{S}^1$, we indicate by $[a,b]$ the closed interval in $\mathbb{S}^1$ obtained by connecting clockwise $a$ and $b$. Let $x$ be a point not within $\alpha$ of $\mathcal{GR}(f)$. Then the following three cases can occur:
\begin{itemize}
\item[$(i)$] The point $x \in [0,2]$ (the right-side semicircle in Figure \ref{Wiseman idea}). Then, recalling that points along the arcs $X_{-k,k}$ move from $K_-$ toward $K_+$, any chain from $x$ back to itself must go counterclockwise to $[-2,0]$ (the left-side semicircle in Figure \ref{Wiseman idea}). If every point of $K_+ \cap [0,x]$ is fixed, then for some metric and some $\varepsilon > 0$ there is no strong $\varepsilon$-chain from $x$ to $[-2,0]$, and thus $x$ cannot be generalized recurrent. Otherwise, there must be a point $k' \in K_+ \cap [0,x]$ that is not fixed. If $g$ moves $k'$ clockwise, then the interval $[k',x]$ maps into its interior. As a consequence, for small enough $\varepsilon > 0$ no $\varepsilon$-chain can get from $[k',x]$ to the left-side semicircle, and so $x$ is not chain recurrent. In particular, $x$ is not generalized recurrent. If $g$ moves $k'$ counterclockwise, then $g$ moves $-k'$ clockwise. Then the interval $[-k',x]$ maps into its interior, and again $x$ is not chain recurrent and therefore not generalized recurrent. 
\item[$(ii)$] The point $x \in [-2,0]$. To return to itself, it would first have to pass through $[0,2]$. Then the argument in $(i)$ shows that $x$ cannot return to $[-2,0]$, and thus $x$ is not generalized recurrent.
\item[$(iii)$] The point $x \in X_{-k,k}$ for some $k \in K_+$. Again, since points along the arcs $X_{-k,k}$ move from $K_-$ toward $K_+$, the only way for the point $x$ to come back to itself is passing through $[0,2]$. Then the argument in $(i)$ shows that $x$ cannot belong to $\mathcal{GR}(g)$.
\end{itemize}
Summarizing, $f\in Hom(X)$ is an example of a homeomorphism such that $\mathcal{GR}(f) \ne \mathcal{CR}(f)$ and $f$ does not permit $\mathcal{GR}$-explosions.

\begin{figure}[h]
\centering
\includegraphics[scale=.3]{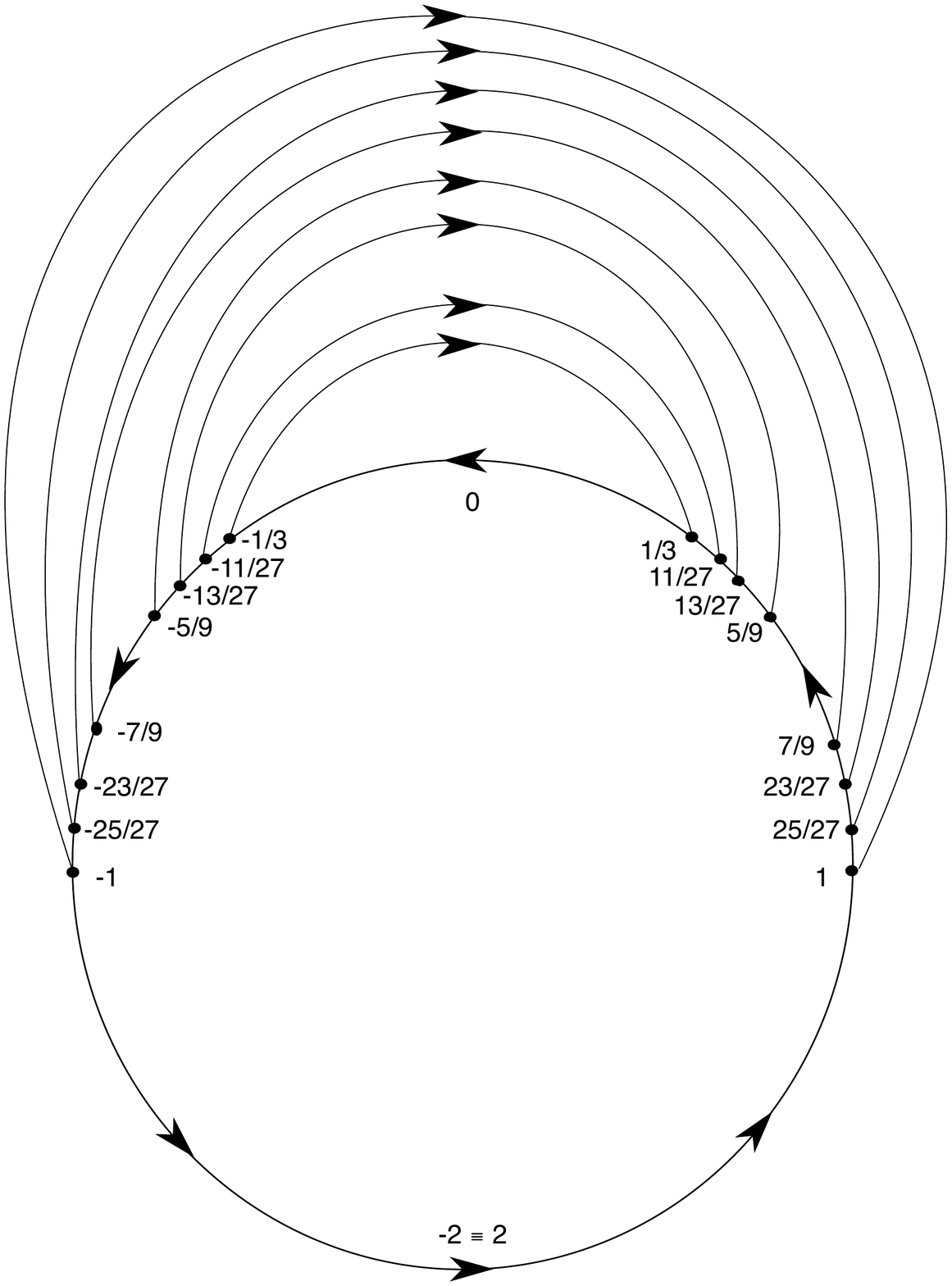}
\caption{The dynamical system of Example \ref{wiseman}.}
\label{Wiseman idea}
\end{figure}
\end{esempio}

\noindent The other implication is valid on a compact topological manifold $M$ with $dim(M) \ge 2$; see Proposition \ref{eccola} below (see also Theorem H in \cite{BF85} for $\mathcal{NW}(f)$). To prove this fact we need the following $\mathscr{C}^0$ closing lemma; see \cite{BF85}[Lemmas 4 and 5] and \cite{NS}[Lemma 13]. 
\begin{lemma} \label{C0} Let $f \in Hom(M)$ be defined on a compact topological manifold $M$ of $dim(M) \ge 2$. If $x \in \mathcal{CR}(f)$ then for any neighborhood $V$ of $f$ in $Hom(M)$ there exists $g \in V$ such that $x \in Per(g)$. 
\end{lemma} 
\begin{proposizione} \label{eccola} Let $f \in Hom(M)$ be defined on a compact topological manifold $M$ of $dim(M) \ge 2$ and such that $\mathcal{GR}(f) \ne M$. $\mathcal{GR}(f) = \mathcal{CR}(f)$ if and only if $f$ does not permit $\mathcal{GR}$-explosions.
\end{proposizione}
\noindent \textit{Proof.} One direction is exactly Corollary \ref{COR UNO}. Assume now that $f$ does not permit $\mathcal{GR}$-explosions, which means that for any open neighborhood $U$ of $\mathcal{GR}(f)$ there exists a neighborhood $V$ of $f$ in $Hom(M)$ such that for any $g\in V$ we have $\mathcal{GR}(g)\subset U$. Then we have:
$$
\mathcal{GR}(f)\subseteq \mathcal{CR}(f) \subseteq \bigcup_{g\in V} Per (g)\subseteq \bigcup_{g\in V}\mathcal{GR}(g)\subseteq U,
$$
where the second inclusion comes from the $\mathscr{C}^0$ closing lemma, here Lemma \ref{C0}. Since the neighborhood $U$ of $\mathcal{GR}(f)$ is arbitrary, we conclude that $\mathcal{GR}(f)=\mathcal{CR}(f)$. \hfill $\Box$ \\
~\newline
\indent The second sufficient condition which avoids $\mathcal{GR}$-explosions is topological stability. We recall the notion of topologically stable homeomorphism, see e.g. \cite{WAL78}[Definition 5] .
\begin{definizione} \label{due a e b} A homeomorphism $f \in Hom(X)$ is topologically stable if there exists a neighborhood $V$ of $f$ in $Hom(X)$ such that for each $g \in V$ there is a continuous function $h_g: X \to X$ satisfying:
\begin{itemize}
\item[$(i)$] $h_g \circ g = f \circ h_g$ (semi-conjugation);
\item[$(ii)$] $h_g \to id$ as $g \to f$ in the uniform topology.
\end{itemize}
\end{definizione}
\begin{proposizione}\label{top stab imply no GR exp}
Let $f\in Hom (X)$ be such that $\mathcal{GR}(f)\neq X$. If $f$ is topologically stable then $f$ does not permit $\mathcal{GR}$-explosions.
\end{proposizione}
\noindent In order to prove Proposition \ref{top stab imply no GR exp}, we need the next result.
\begin{lemma}\label{lemma top stab and GR}
Let $f,g\in Hom(X)$. If there exists a continuous function $h:X\rightarrow X$ such that $$h\circ g=f\circ h$$ 
then $h(\mathcal{GR}(g))\subseteq \mathcal{GR}(f)$.
\end{lemma}
\noindent\textit{Proof.} Recall that, by Theorem \ref{generalized}, 
$$
\mathcal{GR}(f)=\bigcap_{u\in\mathscr{L}(f)}\mathcal{N}(u)
$$
where the intersection is taken over the set $\mathscr{L}(f)$ of continuous Lyapunov functions for $f$. \\
Let us show that if $u\in\mathscr{L}(f)$ then $u\circ h\in\mathscr{L}(g)$. Indeed for any $z\in X$, since $h\circ g=f\circ h$ and $u\in\mathscr{L}(f)$, we have
$$
u\circ h\circ g(z)=u\circ f\circ h(z)\leq u\circ h(z).
$$
Now let $z=h(x)\in h(\mathcal{GR}(g))$ for some $x\in \mathcal{GR}(g)$ and consider $u\in\mathscr{L}(f)$. Since $h\circ g=f\circ h$, $x\in \mathcal{GR}(g)$ and --as remarked above-- $u\circ h\in\mathscr{L}(g)$, we deduce that
$$
u\circ f(z)=u\circ f\circ h(x)=u\circ h\circ g(x)=u\circ h(x)=u(z).
$$
This means that $z=h(x)\in\mathcal{N}(u)$. By the arbitrariness of $u\in\mathscr{L}(f)$ and by Theorem \ref{generalized}, we conclude that $z\in \mathcal{GR}(f)$. Equivalently, $h(\mathcal{GR}(g)) \subseteq \mathcal{GR}(f)$. \hfill $\Box$ \\
~\newline
\noindent We now prove Proposition \ref{top stab imply no GR exp}.\\
~\newline
\noindent\textit{Proof of Proposition \ref{top stab imply no GR exp}.} 
Argue by contradiction and suppose there are an open neighborhood $U$ of $ \mathcal{GR}(f)$ in $X$, a sequence $(g_n)_{n \in \mathbb{N}} \in Hom(X)$ converging to $f$ in the uniform topology and a sequence of points $(x_n)_{n \in \mathbb{N}}$ such that $x_n \in \mathcal{GR}(g_n) \setminus U$. \\
\noindent Since $f$ is topologically stable and the sequence $(g_n)_{n \in \mathbb{N}}$ converges uniformly to $f$, there exists an index $\bar{n}\in\N$ such that 
$$h_n\circ g_n=f\circ h_n \qquad \forall n\geq\bar{n}$$
where $h_n:X\rightarrow X$ is a continuous map. Consequently, by Lemma \ref{lemma top stab and GR}, 
\begin{equation}\label{eq 1 prop top stab}
h_n(\mathcal{GR}(g_n))\subseteq \mathcal{GR}(f) \qquad \forall n \ge \bar{n}.
\end{equation}
Moreover, by Definition \ref{due a e b}, the sequence $(h_n)_{n\in\N}$ converges uniformly to $id$. \\
Define the following continuous function
$$
X\setminus U\ni y\mapsto d(y,\mathcal{GR}(f)) := \min_{x\in \mathcal{GR}(f)} d(y,x).
$$
Since $X\setminus U$ is compact and $(X\setminus U)\cap \mathcal{GR}(f)=\emptyset$, it holds that $\eta := \min_{y\in X\setminus U} d(y,\mathcal{GR}(f)) >0$. \\
Let now $\tilde{n}\in\N$, $\tilde{n}\geq\bar{n}$ be such that 
\begin{equation}\label{eq 2 proof prop top stab}
d_{\mathscr{C}^0}(id,h_n)< \eta \qquad \forall n\geq\tilde{n}.
\end{equation}
On one hand, from \eqref{eq 2 proof prop top stab} we immediately deduce that
\begin{equation} \label{prima dis}
d(x_n,h_n(x_n))< \eta \qquad \forall n\geq\tilde{n}.
\end{equation}
On the other hand, since $x_n\in \mathcal{GR}(g_n)\setminus U$ then $h_n(x_n) \in \mathcal{GR}(f)$ by inclusion \eqref{eq 1 prop top stab}. As a consequence
\begin{equation} \label{seconda dis}
d(x_n,h_n(x_n))\geq \min_{x\in \mathcal{GR}(f)} d(x_n,x)= d(x_n,\mathcal{GR}(f))\geq \min_{y\in X\setminus U} d(y,\mathcal{GR}(f))=\eta.
\end{equation}
Inequalities (\ref{prima dis}) and (\ref{seconda dis}) provide the required contradiction. \hfill $\Box$ \\
~\newline
Clearly, from the previous proposition, we immediately deduce that if $f \in Hom(X)$ with $\mathcal{GR}(f) \ne X$ is topologically stable then $f$ does not permit $\mathcal{GR}$-full explosions. 
As proved by P.\ Walters in \cite{WAL69}[Theorem 1], any Anosov diffeomorphism on a smooth, compact manifold $M$ without boundary is topologically stable. Consequently, the previous proposition applies in particular to every  Anosov diffeomorphism with $\mathcal{GR}(f) \ne M$. \\
~\newline
As with Corollary \ref{COR UNO},  the converse of Proposition \ref{top stab imply no GR exp} is false in general, as shown in the next example.
\begin{esempio}\label{attrattore-repulsore}  On the circle $\mathbb{S}^1 := \R/2\pi\Z$ embedded in $\mathbb{R}^2$ with the usual induced topology, for $n \ge 1$ consider the points
$$
P_n := \left( \cos \left( \frac{\pi}{2^n} \right),  \sin \left( \frac{\pi}{2^n} \right) \right)$$
and
$$Q_n := \left( \cos \left( \frac{\pi}{2^n} \right),  -\sin \left( \frac{\pi}{2^n} \right) \right).$$
Let $f : \mathbb{S}^1 \to \mathbb{S}^1$ be a homeomorphism which fixes exactly $(1,0)$ and every $P_n$ and $Q_n$ and such that:
\begin{itemize}
\item[$(i)$] Every $Q_{n-1}$ and $P_n$ with $n \ge 2$ even is an attractor.
\item[$(ii)$] Every $Q_n$ and $P_{n-1}$ with $n \ge 2$ even is a repeller.
\end{itemize}
We refer to Figure \ref{attr-rep}. In such a case,
$$Fix(f) = \mathcal{GR}(f) = \mathcal{CR}(f).$$
On one hand, since $Fix(f)$ is an infinite set, $f$ is not topologically stable (see \cite{YA}[Theorem 1]). On the other hand, since $\mathcal{GR}(f) = \mathcal{CR}(f)$, $f$ does not permit $\mathcal{GR}$-explosions (see Corollary \ref{COR UNO}). (It is also easy to  verify directly that $f$ does not permit $\mathcal{GR}$-explosions.)
\begin{figure}[h]
\centering
\includegraphics[scale=.25]{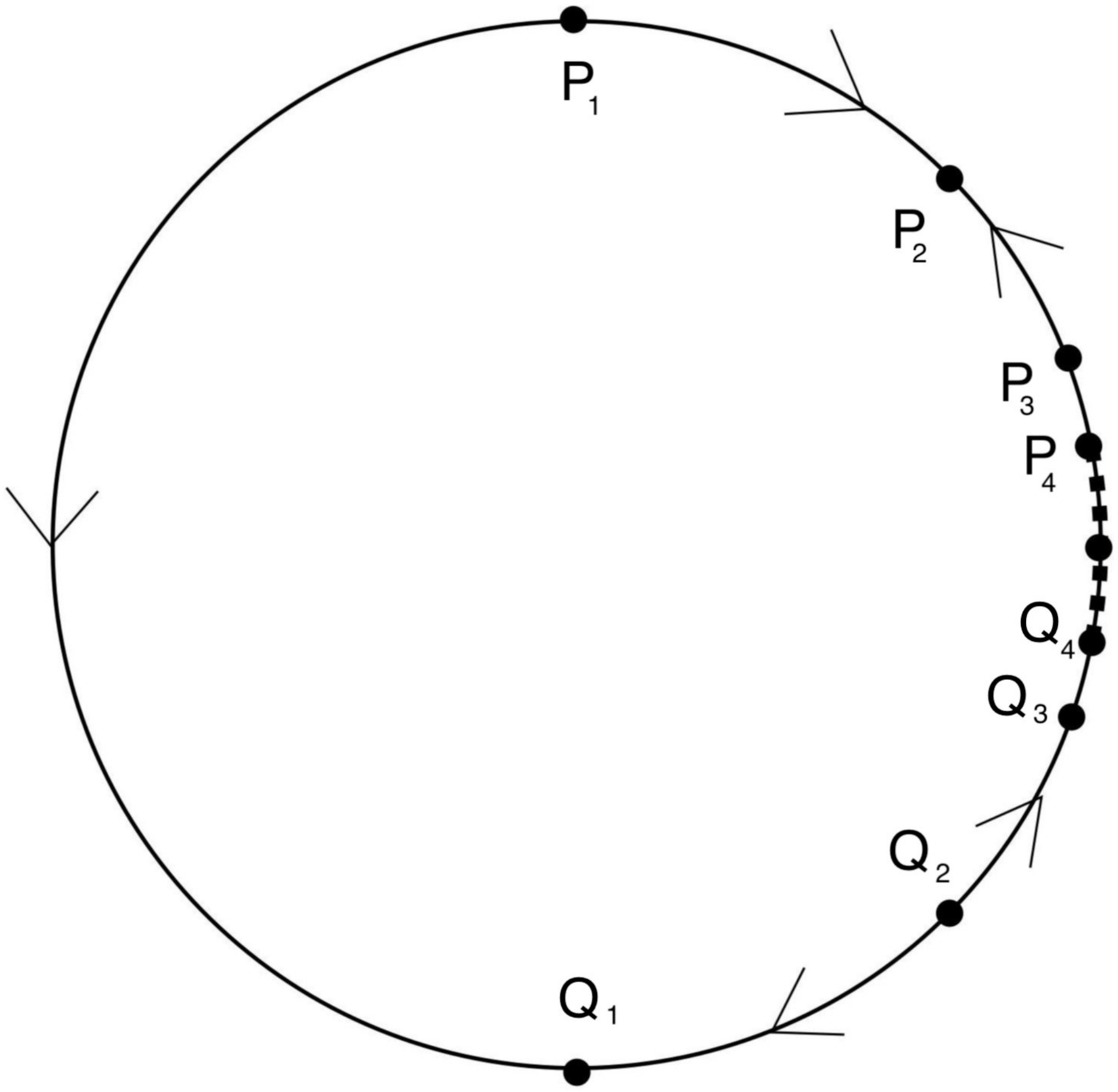}
\caption{The dynamical system of Example \ref{attrattore-repulsore}.}
\label{attr-rep}
\end{figure}
\end{esempio}

\indent We conclude this section by discussing a sufficient condition to avoid $\mathcal{GR}$-full explosions. This condition is an immediate corollary of Theorem \ref{teo 1}. 
\begin{proposizione} \label{anna} Let $f \in Hom(X)$ be such that $\mathcal{GR}(f) \ne X$. If $\mathcal{CR}(f) \ne X$ then $f$ does not permit $\mathcal{GR}$-full explosions. 
\end{proposizione}
\noindent \textit{Proof.} Arguing by contradiction, we assume that $f$ admits $\mathcal{GR}$-full explosions. This means that there exists a sequence $(g_n)_{n\in\N}$ in $Hom(X)$ converging to $f$ in the uniform topology such that $\mathcal{GR}(g_n)=X$ for any $n \in \mathbb{N}$. Consequently, $\mathcal{GR}(g_n)=\mathcal{CR}(g_n)=X$ for any $n \in \mathbb{N}$. Since by hypothesis $\mathcal{CR}(f) \ne X$, we conclude that $f$ permits $\mathcal{CR}$-explosions and this fact contradicts Theorem \ref{teo 1}. \hfill $\Box$ \\ \\
Note that, in general, the converse of Proposition \ref{anna} may be false: in Examples \ref{example 2 JW} and \ref{wiseman} we have defined homeomorphisms on a compact metric space such that $\mathcal{GR}(f) \subsetneq \mathcal{CR}(f) = X$ and $f$ does not permit $\mathcal{GR}$-explosions. In particular, $f$ does not permit $\mathcal{GR}$-full explosions. \\
In the following final example, we slightly modify Example \ref{example 2 JW} in order to obtain $f\in Hom(X)$ such that $\mathcal{GR}(f) \subsetneq \mathcal{CR}(f) = X$ and $f$ does not permit $\mathcal{GR}$-full explosions even though $f$ permits $\mathcal{GR}$-explosions.
\begin{esempio}\label{example JW 2 revised}
	\begin{figure}[h]
		\centering
		\includegraphics[scale=.3]{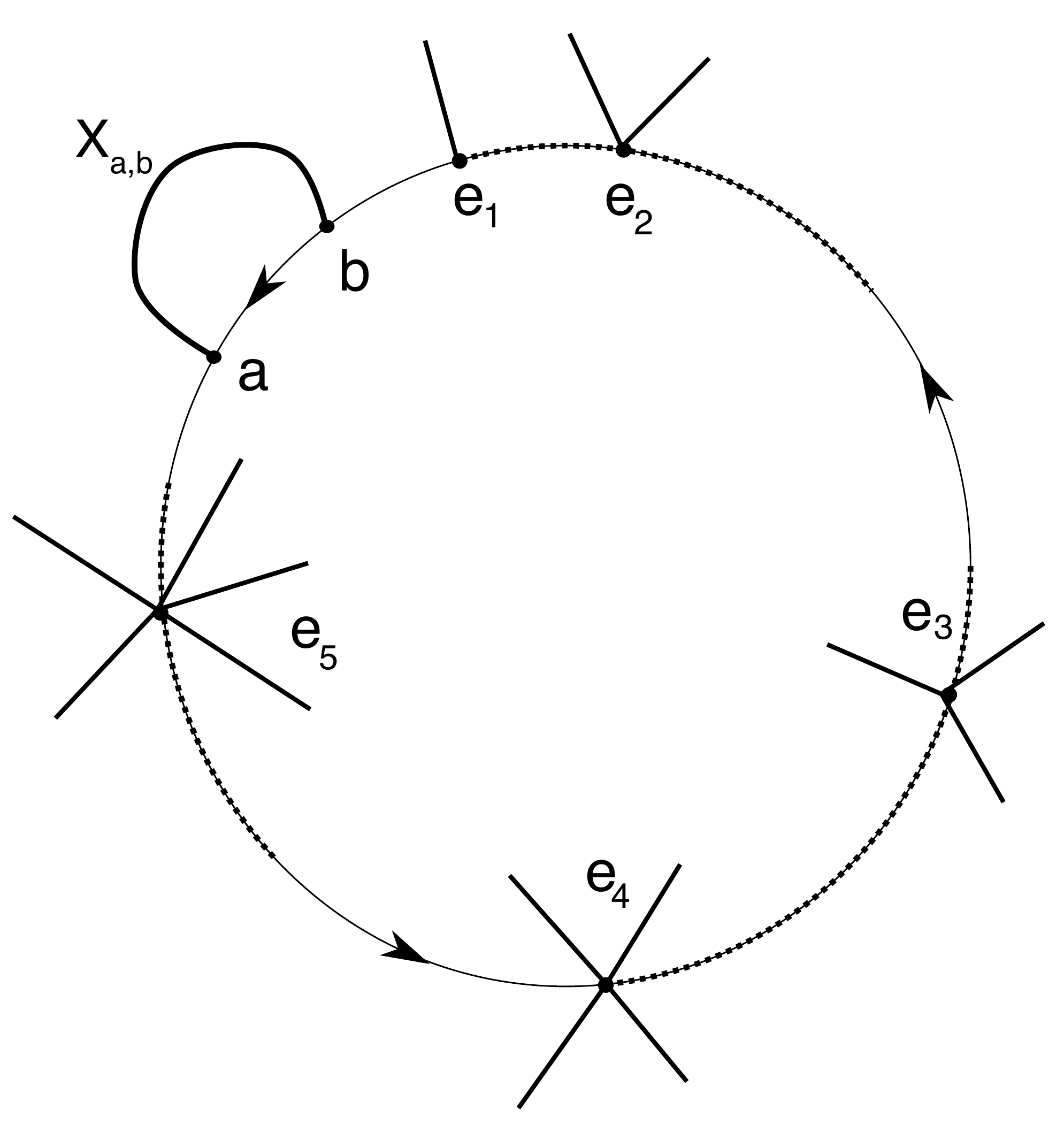}
		\caption{The dynamical system of Example \ref{example JW 2 revised}.}
		\label{Wiseman idea 2}
	\end{figure}
Referring to the compact metric space $X$ constructed in the previous Example \ref{example 2 JW}, let $a,b\in\mathbb{S}^1\setminus K$ be such that the closed interval $[a,b]\subset\mathbb{S}^1$ (obtained by connecting $a$ and $b$ clockwise) does not intersect $K$. Let 
$$X'=X\cup X_{a,b}$$
where $X_{a,b}$ is an arc (disjoint from $X$) from $a$ to $b$.\\
\noindent Define the homeomorphism $f:X'\rightarrow X'$ as follows. $f$ fixes $a,b$, $K$, every copy of the interval $I$ and the arc $X_{a,b}$, while $f$ moves all the points in $\mathbb{S}^1\setminus(K\cup\{a,b\})$ counterclockwise (see Figure \ref{Wiseman idea 2}). \\
\noindent In this case
$$\mathcal{CR}(f)=X'\qquad\text{and}\qquad\mathcal{GR}(f)=Fix(f).$$
Essentially the same argument as in Example \ref{example 2 JW} shows that $f$ does not permit $\mathcal{GR}$-full explosions.\\
\noindent However, $f$ does permit $\mathcal{GR}$-explosions. Indeed, for any $\varepsilon > 0$ there exists $g \in Hom(X')$ such that
$$d_{\mathscr{C}^0}(f,g) < \varepsilon$$
and $g$ modifies the dynamics so that on the arc $X_{a,b}$, points move from $a$ to $b$ and thus points of the open interval $(a,b)$ can return to themselves with arbitrary strong chains. This means that
$$\mathcal{GR}(g) = \mathcal{GR}(f) \cup (a,b).$$
\end{esempio}

\section{$\mathcal{GR}$-explosions and cycles} \label{quattro}
\noindent The goal of this section is to explain the relations between explosions and ``cycles'' for the generalized recurrent set. Let $f \in Hom(X)$ be such that $\mathcal{GR}(f) \ne X$. We start by introducing the notions of decomposition and cycle for $\mathcal{GR}(f)$. \\
~\newline
\indent Given a compact, invariant set $L \subseteq X$, we  define by
$$W^s(L) := \{x \in X: \ d(f^k(x),L) \to 0 \text{ as } k \to +\infty \}$$
and
$$W^u(L) := \{x \in X: \ d(f^{-k}(x),L) \to 0 \text{ as } k \to +\infty \}$$ 
the stable and unstable set, respectively, of $L$ for $f$. In particular, we have that $x \in W^s(L)$ if and only if $\omega(x) \subseteq L$ and $x \in W^u(L)$ if and only if $\alpha(x) \subseteq L$. \\
A decomposition of ${\mathcal{GR}}(f)$ is a finite family $L_1, \ldots , L_k$ of compact, invariant, pairwise disjoint sets in $X$
such that
$${\mathcal{GR}}(f) = \bigcup_{i=1}^k L_i.$$
\begin{definizione} \label{cicli}
Let $L_1, \ldots , L_k$ be a decomposition of ${\mathcal{GR}}(f)$. 
\begin{itemize}
\item[$(i)$] Let $i,j \in \{1, \ldots , k\}$. We write $L_{i}  \rightharpoonup L_{j}$ if
$$\left(W^u(L_i) \cap W^s(L_j)\right) \setminus {\mathcal{GR}}(f) \ne \emptyset.$$
Equivalently, there is a point $x \in X$ outside ${\mathcal{GR}}(f)$ whose orbit is going from $L_i$ to $L_j$. 
\item[$(ii)$] We say that $L_{i_1}, \ldots , L_{i_r}$ form an $r \ge 1$ cycle of $\{L_i\}$ if
$$L_{i_1} \rightharpoonup L_{i_2} \rightharpoonup \ldots \rightharpoonup L_{i_r} \rightharpoonup L_{i_1}.$$
\item[$(iii)$] The decomposition $\{L_i\}$ has no cycles if no subset of $\{L_i\}$ forms an $r \ge 1$ cycle. 
\end{itemize}
\end{definizione}
\indent We first establish that the existence of a decomposition of $\mathcal{GR}(f)$ with no cycles prohibits $\mathcal{GR}$-explosions. This theorem generalizes the corresponding result --due to Pugh and Shub-- for the non-wandering set (see Theorem 6.1 in \cite{PS} for flows and Theorem 5.6 in \cite{WEN} for homeomorphisms). Since in the proof we apply the $\mathscr{C}^0$ closing lemma, the ambient space is a compact, topological manifold $M$ with $dim(M) \ge 2$. 
\begin{teorema} \label{teo 3}  Let $f \in Hom(M)$ be defined on a compact topological manifold $M$ of $dim(M) \ge 2$ and such that $\mathcal{GR}(f) \ne M$. If there exists a decomposition of ${\mathcal{GR}}(f)$ without cycles, then $f$ does not permit ${\mathcal{GR}}$-explosions. \end{teorema}
\noindent The proof of Theorem \ref{teo 3} uses the same techniques of Pugh-Shub and is postponed to Appendix \ref{APP1}. \\
\noindent In order to establish the converse of Theorem \ref{teo 3}, we need to introduce the notion of cycle for a so-called open decomposition of $f \in Hom(X)$ with $\mathcal{GR}(f) \ne X$. \\
~\newline
\indent Given an open set $V \subset X$, we define
\begin{equation} \label{instabile}
V^{u} := \{x \in X : \ \exists\  m\ge 0 \text{ such that }f^{-m}(x) \in V \} = \bigcup_{m \ge 0} f^{m}(V).
\end{equation}
An open decomposition of $f$ is a finite family $V_1, \ldots , V_k$ of open sets in $X$, with pairwise disjoint closures,
such that
$$\mathcal{GR}(f) \subseteq \bigcup_{i=1}^k V_i.$$
\begin{definizione}\label{definizione cicli aperti}
Let $V_1, \ldots , V_k$ be an open decomposition of $f$. 
\begin{itemize}
\item[$(i)$] Let $i,j \in \{1, \ldots , k\}$, $i \ne j$. We write $V_{i}  \ge V_{j}$ if
$$V_j \cap V_i^u \ne \emptyset.$$
Equivalently, there are $x \in V_j$ and $m \ge 0$\footnote{Actually, $m>0$ since the open sets $\{V_i\}$ have pairwise disjoint closures.} such that $f^{-m}(x) \in V_i$.
\item[$(ii)$] We say that $V_{i_1}, \ldots , V_{i_r}$ form an $r > 1$ cycle of $\{V_i\}$ if
$$V_{i_1} \ge V_{i_2} \ge \ldots \ge V_{i_r} \ge V_{i_1}.$$
We say that $V_j$ forms a $1$ cycle of $\{V_i\}$ if there are $x \notin V_j$ and $m, q > 0$ such that
$$f^m(x), f^{-q}(x) \in V_j.$$ 
\item[$(iii)$] The open decomposition $\{V_i\}$ has no cycles if no subset of $\{V_i\}$ forms an $r \ge 1$ cycle. 
\end{itemize}
\end{definizione}
\indent By using the above formalism, we prove the other implication of Theorem \ref{teo 3}.
\begin{teorema} \label{altro verso} Let $f \in Hom(M)$ be defined on a compact topological manifold $M$ of $dim(M) \ge 2$ and such that $\mathcal{GR}(f) \ne M$. If $f$ does not permit $\mathcal{GR}$-explosions, then there exists a decomposition of $\mathcal{GR}(f)$ with no cycles. 
\end{teorema}
\noindent The above theorem is a straightforward consequence of the next one. 
\begin{teorema} \label{altro verso con aperti}Let $f \in Hom(M)$ be defined on a compact topological manifold $M$ of $dim(M) \ge 2$ and such that $\mathcal{GR}(f) \ne M$. If $f$ does not permit ${\mathcal{GR}}$-explosions, then for any open neighborhood $U$ of ${\mathcal{GR}}(f)$ there exists an open decomposition $V_1, \ldots , V_k$ of $f$ with no cycles such that
$${\mathcal{GR}}(f) \subset \bigcup_{i=1}^k V_i \subseteq U.$$
\end{teorema}
\noindent Since the proof of Theorem \ref{altro verso con aperti} follows the same lines of the proof of Lemma 2 in \cite{SS}, it is postponed to Appendix \ref{APP}. \\
~\newline
\noindent\textit{Proof of Theorem \ref{altro verso}.} Arguing by contradiction, we suppose that every decomposition $L_1, \ldots , L_k$ admits an $r \ge 1$ cycle
$$L_{i_1} \rightharpoonup L_{i_2} \rightharpoonup \ldots \rightharpoonup L_{i_r} \rightharpoonup L_{i_1}.$$
This means that there are $x_1, \ldots , x_r \notin {\mathcal{GR}}(f)$ such that
$$\alpha(x_{j}) \subseteq L_{i_j}, \qquad \omega(x_{j}) \subseteq L_{i_{j+1}} \qquad \forall j \in \{1, \ldots r-1\}$$
and
$$\alpha(x_{r}) \subseteq L_{i_r}, \qquad \omega(x_{r}) \subseteq L_{i_{1}}.$$
Since the $\{ L_i \}$ are compact and pairwise disjoint, we can choose a family $\{\mathcal{L}_i\}$ of open sets, with pairwise disjoint closures such that $L_i \subset \mathcal{L}_i$ for any $i = 1, \ldots , k$. Let us consider the open neighborhood  
$$U := \bigcup_{i=1}^{k} \mathcal{L}_i.$$ 
of ${\mathcal{GR}}(f)$. By shrinking each $\mathcal{L}_i$ a bit, we can assume that the points $x_1, \ldots , x_r \notin U$.  Then, by construction, any open decomposition $\{V_i\}$ of $f$ such that $\mathcal{GR}(f)\subset \bigcup_i V_i\subseteq U$ has cycles. This fact contradicts Theorem \ref{altro verso con aperti} applied to $U$ and concludes the proof. \hfill $\Box$ \\
~\newline
Finally, Theorems \ref{teo 3} and \ref{altro verso} give us the equivalence between no ${\mathcal{GR}}$-explosions and the existence of a decomposition of $\mathcal{GR}(f)$ with no cycles.
\begin{teorema}\label{corollario iff} Let $f \in Hom(M)$ be defined on a compact topological manifold $M$ of $dim(M) \ge 2$ and such that $\mathcal{GR}(f) \ne M$. $f$ does not permit ${\mathcal{GR}}$-explosions if and only if there exists a decomposition of ${\mathcal{GR}}(f)$ with no cycles. 
\end{teorema}

\begin{remark}
We notice that in the proof of the previous theorem --see Appendixes \ref{APP1} and \ref{APP}-- we do not use the full definition of the generalized recurrent set, but only the fact that $\mathcal{GR}(f)$ is a compact, invariant set, which contains $\mathcal{NW}(f)$ and is contained in $\mathcal{CR}(f).$ Consequently, with the same proof, we can obtain the corresponding results both for $\mathcal{SCR}(f)$ and $\mathcal{CR}(f)$.
\end{remark}
\begin{teorema} \label{teo 4}  Let $f \in Hom(M)$ be defined on a compact topological manifold $M$ of $dim(M) \ge 2$ and such that $\mathcal{SCR}(f) \ne M$. $f$ does not permit ${\mathcal{SCR}}$-explosions if and only if there exists a decomposition of ${\mathcal{SCR}}(f)$ with no cycles.  
\end{teorema}
\begin{remark} \label{rem CR}
The corresponding result holds for $\mathcal{CR}(f)$; since $\mathcal{CR}$-explosions cannot occur (see Theorem \ref{teo 1}), we see that for $f \in Hom(M)$  defined on a compact topological manifold $M$ of $dim(M) \ge 2$, there exists a decomposition of $\mathcal{CR}(f)$ with no cycles.  In fact, this holds for compact metric spaces, as a consequence of the existence of a complete Lyapunov function.
\end{remark}

{\section{Applications to strict Lyapunov functions and genericity} \label{cinque}

Let $X$ be a compact metric space and $f \in Hom(X)$. We say that a continuous Lyapunov function $u:X\rightarrow\R$ for $f$ is strict if it is not a first integral. Equivalently, this means that $\mathcal{N}(u) \ne X$, i.e.\ there exists $x \in X$ such that
$$u(f(x)) < u(x).$$
By Theorem \ref{generalized}, $f$ admits a continuous strict Lyapunov function if and only if $\mathcal{GR}(f) \ne X$. In this section, we collect the results of Sections \ref{tre} and \ref{quattro} in order to give conditions to give an affirmative answer to this question: \\
~\newline
\textit{For a given homeomorphism, is the property of admitting a continuous strict Lyapunov function stable under $\mathscr{C}^0$ perturbations?} \\
~\newline
\noindent From Corollary \ref{COR UNO}, Proposition \ref{top stab imply no GR exp} and Proposition \ref{anna}, we deduce the following
\begin{proposizione}\label{Lyapunov stability}
Let $f\in Hom(X)$ admit a continuous strict Lyapunov function. Suppose that one of these hypotheses holds: 
\begin{itemize}
\item[$(i)$] $\mathcal{GR}(f) = \mathcal{CR}(f)$;
    \item[$(ii)$] $\mathcal{CR}(f)$ is strictly contained in $X$;
    \item[$(iii)$] $f$ is topologically stable.
    \end{itemize} 
Then there exists a neighborhood $V$ of $f$ in $Hom(X)$ such that any $g \in V$ admits a continuous strict Lyapunov function.    
\end{proposizione}       
\noindent Moreover, from Theorem \ref{corollario iff}, we obtain that a sufficient condition is the existence of a decomposition of $\mathcal{GR}(f)$ without cycles.
\begin{proposizione} \label{per varieta}
Let $f\in Hom(M)$ admit a continuous strict Lyapunov function. If there exists a decomposition of $\mathcal{GR}(f)$ without cycles then there exists a neighborhood $V$ of $f$ in $Hom(M)$ such that any $g \in V$ admits a continuous strict Lyapunov function.
\end{proposizione}
\indent  We finally remark that the property of admitting a continuous strict Lyapunov function is generic. For this purpose, let $M$ be a smooth, compact manifold with metric $d$. We endow $Hom(M)$ with the metric
$$d_{\mathscr{H}}(f,g) = \max_{x \in M} \left(\max \left[ d(f(x),g(x)) , d(f^{-1}(x),g^{-1}(x))\right]\right).$$
With this metric, $Hom(M)$ is a complete space and therefore it is a Baire space.
A property in $Hom(M)$ is
said to be generic if the set of $f \in Hom(M)$ satisfying this property contains a residual set, i.e. a countable intersection of open dense sets.
\begin{proposizione}\label{Lyapunov genericity} On a smooth, compact manifold $M$, the property in $Hom(M)$ of admitting a continuous strict Lyapunov function is generic. 
\end{proposizione}
\begin{proof}
By Theorem 3.1 in \cite{HuAttr}, the property in $Hom(M)$ of having $int(\mathcal{CR}(f)) = \emptyset$ is generic and so also that of having $\mathcal{GR}(f)\neq M$. Therefore, by Theorem \ref{generalized}, the property in $Hom(M)$ of admitting a continuous strict Lyapunov function is generic too.
\end{proof}}


\appendix

\section{Proof of Theorem \ref{teo 3}} \label{APP1}
Arguing by contradiction, we suppose that there are an open neighborhood $U$ of ${\mathcal{GR}}(f)$ in $M$, a sequence $(h_n)_{n\in\N} \in Hom(M)$ converging to $ f$ in the uniform topology and a sequence of points $(a_n)_{n\in\N} \in {\mathcal{GR}}(h_n) \setminus U$.  Since $X$ is compact and $U$ is open, we can assume that $ a_n \to b_1\notin \mathcal{GR}(f)$. By assumption
$$a_n \in \mathcal{GR}(h_n) \subseteq \mathcal{CR}(h_n)$$ and therefore --by the $\mathscr{C}^0$ closing lemma-- for each $n \in \mathbb{N}$ there exists $g_n \in Hom(M)$ such that
$$d_{\mathscr{C}^0} (h_n,g_n) < \frac{1}{n} \qquad \text{and} \qquad a_n \in Per(g_n).$$
Denote by $T_n \ge 1$ the least period of $a_n$ 
and define
$$k_n := n \, T_n.$$
Clearly, $k_n \rightarrow +\infty$ and
$$g_n^{k_n}(a_n) = a_n \qquad \forall n \in \mathbb{N}.$$
This means that there exist $g_n \to f$ in the $\mathscr{C}^0$ topology and $k_n \to +\infty$ such that every point $a_n$ can be equivalently represented by $g_n^{k_n}(a_n)$. Recall that
$$g_n^{k_n}(a_n) = a_n \to b_1 \notin \mathcal{GR}(f).$$
\indent The alpha limit and the omega limit of every point of $M$ are contained in $\mathcal{NW}(f) \subseteq \mathcal{GR}(f)$  (see Proposition 3.3.4 in \cite{K}). 
 In particular, $\alpha(b_1),  \omega(b_1) \subseteq L_1 \cup \ldots \cup L_k$ and therefore, by Theorem 5.4 in \cite{WEN}, 
 there exist $i_0, i_1 \in \{1, \ldots, k\}$ such that $\alpha(b_1) \subseteq L_{i_0}$ and $\omega(b_1) \subseteq L_{i_1}$ (equivalently $b_1 \in W^{u}(L_{i_0}) \cap W^{s}(L_{i_1})$). Since we have no cycles $i_0 \ne i_1$. \\
\indent Take a compact neighborhood $U_{i_1}$ of $L_{i_1}$ such that
$$b_1 \notin U_{i_1} \qquad \text{and} \qquad U_{i_1} \cap L_j = \emptyset \qquad \forall j \ne {i_1}.$$
Equivalently, since the $L_j$ are invariant,
$$b_1 \notin U_{i_1} \qquad \text{and} \qquad f(U_{i_1}) \cap L_j = \emptyset \qquad \forall j \ne {i_1}.$$
As in \cite{WEN}[Page 147], we denote pieces of $g_n$-orbits from $a_n$ to itself as:
$$[a_n,a_n] := (a_n, g_n(a_n), g_n^2(a_n), \dots , g_n^{k_n}(a_n)),$$
$$[a_n,a_n) := (a_n, g_n(a_n), g_n^2(a_n), \dots , g_n^{k_n-1}(a_n)),$$
$$(a_n,a_n) := (g_n(a_n), g_n^2(a_n), \dots , g_n^{k_n-1}(a_n)).$$	
Since $a_n \to b_1 \in W^s(L_{i_1})$, $g_n \to f$ in the uniform topology and $k_n \to +\infty$, there exists a sequence of points $p_n \in [a_n, a_n]$ such that
$$d(p_n, L_{i_1}) \to 0,\qquad\text{for } n \to +\infty.$$
In particular, since $b_1 \notin U_{i_1}$, $p_n \in (a_n,a_n)$ for large $n$. Moreover, for the same reason, for $n$ sufficiently large, there is $m_n \ge 1$ such that
$$p_n, g_n(p_n), \ldots ,g_n^{m_n - 1}(p_n) \in int \ U_{i_1}$$
and
$$g_n^{m_n}(p_n) \notin int \ U_{i_1}.$$
Then
$$z_n := g_n^{m_n}(p_n) \in g_n(int \ U_{i_1}) \setminus int \ U_{i_1} \subset g_n(U_{i_1}) \setminus int \ U_{i_1}.$$
Note that, since $d(p_n, L_{i_1}) \to 0$, $g_n \to f$ in the $\mathscr{C}^0$ topology and $L_{i_1}$ is $f$-invariant, we have
$$m_n \to +\infty.$$
\noindent Finally, since $g_n \to f$ in the uniform norm, we have that for $n$ large enough
$$g_n(U_{i_1}) \cap L_j = \emptyset \qquad \forall j \ne i_{1}.$$
Let $b_2$ be a limit point of the sequence $(z_n)_{n\in\N}$. Hence $b_2 \notin \mathcal{GR}(f)$. Moreover, since the first $m_n$ iterates of $z_n$ with respect to $g_n^{-1}$ are contained in $U_{i_1}$, $g_n \to f$ in the uniform norm and $m_n \to +\infty$, it follows that $f^{-m}(b_2) \in U_{i_1}$ for all $m \ge 1$ (recall that, since $X$ is compact, it holds also that $g_n^{-1}\rightarrow f^{-1}$ in the $\mathscr{C}^0$ topology). This means that $\alpha(b_2) \subseteq L_{i_1}$ or equivalently $b_2 \in W^{u}(L_{i_1})$. Moreover, from the hypothesis that there are no cycles, $\omega(b_2) \subseteq L_{i_2}$ with $i_2 \ne i_0, i_1$. \\
\indent Recall now that $z_n \in (a_n,a_n]$. This means that there exists $r_n \ge 1$ such that $z_n$ (which is the first point of the $g_n$-orbit of $p_n$ outside $int \ U_{i_1}$) can be represented as
$$z_n = g_n^{r_n}(a_n).$$ In order to proceed similarly with $b_2 \notin \mathcal{GR}(f)$, we need to prove that also the sequence 
$$k_n - r_n \to +\infty.$$ 
Suppose to the contrary that $k_n - r_n$ is uniformly bounded. This means that
$$b_1 = \lim_{n \to +\infty} g_n^{k_n - r_n}(z_n)$$
has the same alpha limit of $b_2 = \lim_{n \to +\infty} z_n$, that is $\alpha(b_1) \subseteq L_{i_1}$. Since $\alpha(b_1) \subseteq L_{i_0}$ with $i_1 \ne i_0$, this is the desired contradiction. In particular, $z_n \in (a_n,a_n)$ for large $n$. \\
\indent We now apply the same argument to $b_2 \notin \mathcal{GR}(f)$. Take a compact neighborhood $U_{i_2}$ of $L_{i_2}$ such that 
$$b_2 \notin U_{i_2} \qquad \text{and} \qquad f(U_{i_2}) \cap L_j = \emptyset \qquad \forall j \ne {i_2}.$$
In order to continue, we denote pieces of $g_n$-orbits from $z_n = g_n^{r_n}(a_n)$ to $a_n$ as:
$$[z_n,a_n] := (z_n, g_n(z_n), g_n^2(z_n), \dots , g_n^{k_n-r_n}(z_n)),$$
$$[z_n,a_n) := (z_n, g_n(z_n), g_n^2(z_n), \dots , g_n^{k_n -r_n -1}(z_n)),$$
$$(z_n,a_n) := (g_n(z_n), g_n^2(z_n), \dots , g_n^{k_n-r_n-1}(z_n)).$$	
As in the previous case, 
there exists a sequence of points $p'_n \in [z_n, a_n]$ such that
$$d(p'_n, L_{i_2}) \to 0,\qquad\text{for } n \to +\infty.$$
In particular, since $b_2 \notin U_{i_2}$, $p'_n \in (z_n,a_n)$ for large $n$. Moreover, 
for $n$ sufficiently large, there is $m'_n \ge 1$ such that
$$p'_n, g_n(p'_n), \ldots ,g_n^{m'_n - 1}(p'_n) \in int \ U_{i_2}$$
and
$$g_n^{m'_n}(p'_n) \notin int \ U_{i_2}.$$
Then
$$z'_n := g_n^{m'_n}(p'_n) \in g_n(int \ U_{i_2}) \setminus int \ U_{i_2} \subset g_n(U_{i_2}) \setminus int \ U_{i_2}$$
and
$$m'_n \to +\infty.$$
\noindent Finally, since $g_n \to f$ in the uniform norm, we have that for $n$ large enough
$$g_n(U_{i_2}) \cap L_j = \emptyset \qquad \forall j \ne {i_2}.$$
Let $b_3$ be a limit point of the sequence $(z'_n)_{n\in\N}$. Hence, $b_3 \notin \mathcal{GR}(f)$
and
$\alpha(b_3) \subseteq L_{i_2}$ 
Moreover, since there are no cycles, $\omega(b_3) \subseteq L_{i_3}$ with $i_3 \ne i_0, i_1,i_2$. \\
\indent We finally notice that there exists $r'_n \ge 1$ such that $z'_n$ can be represented as
$$z'_n = g_n^{r'_n}(a_n).$$
Arguing as for $k_n - r_n$, it can be shown that the sequence
$k_n - r'_n \to +\infty.$
In particular, $z'_n \in (z_n,a_n)$ for large $n$. \\ 
\indent We can proceed with $b_3$ exactly as $b_2$ and produce, by iteration, a chain of $L_j$'s with no repetitions and length greater that $k$. This is the desired contradiction and the theorem is proved. \hfill $\Box$

\section{Proof of Theorem \ref{altro verso con aperti}} \label{APP}

\noindent We begin with a definition and a technical lemma. 
\begin{definizione}  \label{def GR} Let $f \in Hom(X)$, $y,z \in X$ and $\varepsilon > 0$. A ${\mathcal{GR}}$-chain of $\varepsilon$ balls of $f$ from $y$ to $z$ is a finite sequence $\{B_i\} = B_1, \ldots , B_n$ of open convex balls such that
\begin{itemize}
\item[$(i)$] $y \in B_1$ and $z \in B_n;$
\item[$(ii)$] For any $i \in \{1, \ldots , n\}$, $diam(B_{i}) < \varepsilon$ and $B_{i} \cap {\mathcal{GR}}(f) \ne \emptyset;$
\item[$(iii)$] For any $i \in \{1, \ldots , n-1\}$, there exists $m(i) \ge 0$  such that $f^{m(i)}(B_i) \cap B_{i+1} \ne \emptyset$. 
\end{itemize}
\end{definizione}
\noindent The same definition of chain of balls is given in \cite{SS}[Page 590] for the non-wandering set. Moreover, Lemma \ref{tecnico} below is a generalization of Lemma 3 in \cite{SS}.

\begin{lemma} \label{tecnico} Let $f \in Hom(M)$ be defined on a compact topological manifold $M$ of $dim(M) \ge 2$. Let $y,z \in M$ and let $\varepsilon > 0$. Given a {$\mathcal{GR}$}-chain of $\varepsilon$ balls of $f$ from $y$ to $z$, there exists $g \in Hom(M)$ such that:
\begin{itemize}
\item[$(a)$] $d_{\mathscr{C}^0}(f,g) < 4\pi \varepsilon;$
\item[$(b)$] $g^N(f^{-1}(y)) = f(z)$ for some $N > 0;$
\item[$(c)$] $g=f$ outside the union of the $\varepsilon$ balls of the $\mathcal{GR}$-chain. 
\end{itemize}
\end{lemma}
\noindent \textit{Proof of Lemma \ref{tecnico}.} Let $\{B_i\} = B_1, \ldots , B_n$ be a $\mathcal{GR}$-chain of $\varepsilon$ balls from $y$ to $z$. For any $i \in \{1, \ldots , n-1\}$, let $m(i) \ge 0$ be the first integer such that $f^{m(i)}(B_i)\cap B_{i+1}\neq\emptyset$. For any $i \in \{1, \ldots , n-1\}$, choose 
$$w_{i+1} \in f^{m(i)}(B_i)\cap B_{i+1}$$ 
and denote 
$$z_i := f^{-m(i)}(w_{i+1}) \in B_i$$
\noindent Moreover, since $\{B_i\}$ is a ${\mathcal{GR}}$-chain, for any $i \in \{1,\dots,n\}$ we can select 
$$k_i \in B_i\cap {\mathcal{GR}}(f).$$ In particular, each $k_i \in \mathcal{CR}(f)$.
Consequently, let $(r_{i,1}=k_i,r_{i,2},\dots,r_{i,N(k_i)} = k_i)$ be a $\varepsilon$-chain from $r_{i,1}=k_i$ to itself. \\
\indent Let now consider the following couples of points:  
\begin{flalign*}
&& (y , k_1)  && 
\end{flalign*}
\begin{flalign*}
&&  (f(r_{1,j}) , r_{1,j+1}) && \mathllap{\forall j=1,\dots,N(k_1)-2}
\end{flalign*}
\begin{flalign*}
&& (f(r_{1,N(k_1)-1}),z_1)  && 
\end{flalign*}
\begin{flalign*}
&& (f^j(z_1),f^j(z_1)) && \mathllap{\forall j=1,\dots,m(1)-1}
\end{flalign*}
\begin{flalign*}
&& (f^{m(1)}(z_1),k_2) = (w_2,k_2) &&
\end{flalign*}
\begin{flalign*}
&& (f(r_{2,j}),r_{2,j+1}) &&\mathllap{\forall j=1,\dots,N(k_2)-2}
\end{flalign*}
\begin{flalign*}
&& (f(r_{2,N(k_2)-1}),z_2) &&
\end{flalign*}
\begin{flalign*}
&& (f^j(z_2),f^j(z_2)) &&\mathllap{\forall j=1,\dots,m(2)-1}
\end{flalign*}
\noindent and so on, until
\begin{flalign*}
&& (f^{m(n-1)}(z_{n-1}), z) = (w_n,z) &&
\end{flalign*}
\begin{flalign*}
&& (f(z),f(z)). &&
\end{flalign*}

By perturbing $f$ a little, we can assume that for every $(q_i,p_i)$, $(q_j,p_j)$ with $i \ne j$, it holds that $q_i \ne q_j$ and $p_i \ne p_j$. Moreover, by construction, for any couple $(q,p)$ given above, the distance $d(q,p)<2\epsilon$. Consequently, applying Lemma 13 in \cite{NS}, we obtain a homeomorphism $\eta:M\rightarrow M$ such that 
$$d_{{\mathscr{C}}^0}(\eta\circ f,f)<4 \pi \varepsilon \qquad \text{and} \qquad \eta(q)=p$$
for any such a couple $(q,p)$. \\
\indent We finally prove that there exists $N>0$ such that $(\eta\circ f)^N(f^{-1}(y))=f(z)$. Indeed
$$
(\eta\circ f) (f^{-1}(y))=k_1,\qquad (\eta\circ f) (k_1)=r_{1,2}, \qquad (\eta \circ f)(r_{1,2})=r_{1,3}, \quad \dots \quad (\eta \circ f)(r_{1,N(k_1)-1})=z_1
$$
$$
(\eta\circ f)(z_1)=f(z_1), \qquad (\eta \circ f) (f(z_1)) = f^2(z_1), \quad\dots\quad (\eta\circ f) (f^{m(1)-1}(z_1))=k_2
$$
and so on, until 
$$
(\eta\circ f)(f^{m(n-1)-1}(z_{n-1}))=\eta(w_n)=z, \qquad \eta\circ f(z)=f(z).
$$
\noindent This proves that, after a number $N>0$ of iterations of $\eta\circ f$, we have
$$
(\eta\circ f)^N(f^{-1}(y))=f(z)
$$
and so $g:=\eta\circ f$ is the desired perturbation.
\hfill\qed \\
~\newline
\noindent We finally prove Theorem \ref{altro verso con aperti} by using the same techniques of Lemma 2 in \cite{SS}. \\
~\newline
\noindent \textit{Proof of Theorem \ref{altro verso con aperti}.} Let $U$ be an arbitrary open neighborhood of ${\mathcal{GR}}(f)$ in $M$. Since $f$ does not permit ${\mathcal{GR}}$-explosions, there exists $\varepsilon>0$ such that if the ${\mathscr{C}}^0$ distance from $g\in Hom(M)$ to $f$ is less than $4 \pi\varepsilon$ then 
	\begin{equation}\label{relation f g U}
	{\mathcal{GR}}(g)\subset U.
	\end{equation}
\indent Let $\{B_{\alpha}\}$ be a finite covering of ${\mathcal{GR}}(f)$ by open convex balls such that, for any $\alpha$, 
$$B_{\alpha} \subseteq U \quad \text{ and } \quad diam(B_{\alpha}) < \varepsilon.$$
    We denote by $\{U_{\beta}\}$ the connected components of $\bigcup_{\alpha} B_{\alpha}$. By shrinking the balls $\{B_{\alpha}\}$ if necessary, we can assume that the $\{U_{\beta}\}$ have pairwise disjoint closures. We now introduce the following equivalence relation on pairs of $\{U_{\beta}\}$. We say that $U_{\beta_1}$ is related to $U_{\beta_2}$ if either $U_{\beta_1} = U_{\beta_2}$ or there is a common cycle containing both $U_{\beta_1}$ and $U_{\beta_2}$, according to Definition \ref{definizione cicli aperti}. Moreover, we indicate by $\tilde{U}_{i}$  the union of the members of the same relation class. By construction, $\{\tilde{U}_i\}$ is an open decomposition of $f$ with no $r$ cycles for $r>1$. In order to obtain an open decomposition of $f$ with no $r$ cycles for $r\geq 1$, we define
$$\tilde{U}_i^s := \{x \in X : \ \exists  \ q\ge 0 \text{ such that }f^{q}(x) \in \tilde{U}_i \}$$
and (see also formula (\ref{instabile}))
$$V_i := \tilde{U}_i^s \cap \tilde{U}_i^u.$$
Then $V_1, \ldots , V_k$ is an open decomposition of $f$ with no cycles. That is, $\{V_i\}$ is a finite family of open sets with pairwise disjoint closures such that 
${\mathcal{GR}}(f) \subseteq \bigcup_{i = 1}^k V_i$
and no subset of $\{V_i\}$ forms an $r \ge 1$ cycle. \\
\indent It remains to prove that $V_i \subseteq U$ for any $i$. Let $i$ be fixed. Arguing by contradiction, we suppose there exists 
$$x \in V_i \setminus U.$$
Since $V_i = \tilde{U}_i^s \cap \tilde{U}_i^u$, there exist $m,q > 0$ such that 
$$y := f^m(x) \in \tilde{U}_i \quad \text{ and } \quad z:= f^{-q}(x) \in \tilde{U}_i$$
Let $m,q > 0$ be the minimal integers with this property. Recall that, by definition, $\tilde{U}_i$ is the union of the members of $\{B_{\alpha}\}$ in the same relation class. This means that there exists a ${\mathcal{GR}}$-chain of $\varepsilon$ balls of $f$ from $y$ to $z$ and therefore we can apply Lemma \ref{tecnico}.  Let $g \in Hom(M)$ be the homeomorphism given by Lemma \ref{tecnico}. Since $m,q > 0$ are the minimal integers such that $y = f^m(x) \in \tilde{U}_i \text{ and } z= f^{-q}(x) \in \tilde{U}_i$, the points
$$f^i(x) \qquad \forall i \in \{0, \ldots , m-1\}$$
and
$$f^{-j}(x)  \qquad \forall j \in \{0, \ldots q-1\}$$
are outside the union of the $\varepsilon$ balls of the ${\mathcal{GR}}$-chain. Consequently, by point $(iii)$ of Lemma \ref{tecnico}, 
\begin{equation}\label{prima}
g(f^i(x)) = f(f^i(x)) \qquad \forall i \in \{0, \ldots , m-1\}
\end{equation}
and
$$g(f^{-j}(x)) = f(f^{-j}(x))  \qquad \forall j \in \{0, \ldots q-1\}.$$
Since $x = f^{-m}(y) = f^q(z)$, the previous equalities become respectively
$$
g(f^{-i}(y)) = f(f^{-i}(y)) \qquad \forall i \in \{1, \ldots , m\}
$$
and
\begin{equation} \label{seconda}
g(f^{j}(z)) = f(f^{i}(z))  \qquad \forall j \in \{1, \ldots q\}.
\end{equation}
Consequently, by (\ref{seconda}):
\begin{eqnarray*}
x = f^q(z) = f(f^{q-1}(z)) = g(f^{q-1}(z)) = g^2(f^{q-2}(z)) = \ldots = g^{q-1}(f(z)).
\end{eqnarray*}
Moreover, by point $(ii)$ of Lemma \ref{tecnico}, $g^N(f^{-1}(y))= f(z)$ for some $N > 0$ and therefore (see also (\ref{prima})) 
\begin{eqnarray*}
x &=& g^{N + q -1}(f^{-1}(y)) = g^{N + q -1} (f^{m-1}(x)) =  g^{N + q -1} (f(f^{m-2}(x))) = g^{N + q}(f^{m-2}(x)) \\
&=& \ldots = g^{N+q+m-2}(x).
\end{eqnarray*}
This means that $x \in Per(g)$. \\
Recall that, by property $(i)$ of Lemma \ref{tecnico}, $g \in Hom(M)$ is such that $d_{\mathscr{C}^0}(f,g) < 4 \pi \varepsilon$. Consequently, by \eqref{relation f g U}, ${\mathcal{GR}}(g) \subset U$. Since $Per(g)\subseteq {\mathcal{GR}}(g)$, the point $x$ should belong to $U$ and this gives us the required contradiction.
\hfill  $\Box$ \\

\end{document}